\documentclass{amsart}
\usepackage{amsmath}
\usepackage{amssymb}
\usepackage{amscd}
\usepackage[all]{xy}
\usepackage{color}
\usepackage{tikz}
\usetikzlibrary{shapes,arrows.meta, positioning}
\newtheorem{theorem}{Theorem}[section]
\newtheorem{lemma}[theorem]{Lemma}
\newtheorem{corollary}[theorem]{Corollary}
\newtheorem{proposition}[theorem]{Proposition}

\newtheorem{example}[theorem]{Example}
\newtheorem{remark}[theorem]{Remark}

\mathchardef\mhyphen="2D

\begin{document}
\title[$S$-Noetherian modules and $S$-strong Mori modules]
{On $S$-Noetherian modules and $S$-strong Mori modules}

\author [H. Baek] {Hyungtae Baek}
\address{(Baek) School of Mathematics,
Kyungpook National University, Daegu 41566,
Republic of Korea}
\email{htbaek5@gmail.com}

\author [J.W. Lim] {Jung Wook Lim}
\address{(Lim) Department of Mathematics,
College of Natural Sciences,
Kyungpook National University, Daegu 41566,
Republic of Korea}
\email{jwlim@knu.ac.kr}

\thanks{Words and phrases: $S$-Noetherian module,
$S$-strong Mori module, Hilbert basis theorem, ($t$-)Nagata module}

\thanks{$2020$ Mathematics Subject Classification: 13A15, 13B25, 13E99}

\begin{abstract}
In this paper, we study some properties of
$S$-Noetherian modules and $S$-strong Mori modules.
Among other things, we prove the Hilbert basis theorem for
$S$-Noetherian modules and $S$-strong Mori modules.
\end{abstract}

\maketitle

\section{Introduction}

In this paper,
$R$ always denotes a commutative ring with identity,
$S$ is a (not necessarily saturated) multiplicative subset of $R$
and $M$ stands for a unitary $R$-module.
(For the sake of avoiding the confusion,
we use $D$ instead of $R$ when $R$ is an integral domain.)

Recall that $M$ is a {\it Noetherian module} if
every submodule of $M$ is finitely generated
(or equivalently, the ascending chain condition on submodules of $M$ holds)
and $R$ is a {\it Noetherian ring} if
$R$ is a Noetherian $R$-module.
In \cite{wang1}, Wang and McCasland introduced new algebraic objects
whose classes contain those with Noetherian property.
They defined a $w$-module $M$ to be a {\it strong Mori module} (SM-module) if
$M$ satisfies the ascending chain condition on $w$-submodules of $M$
(or equivalently, each $w$-submodule of $M$ is $w$-finite),
where $w$ denotes the so-called $w$-operation on $M$.
(Recall that a $w$-module $M$ is {\em $w$-finite} if
there exists a finitely generated submodule $F$ of $M$
such that $M = F_w$.)
Also, $D$ is said to be a {\em strong Mori domain} (SM-domain) if
$D$ is an SM-module as a $D$-module.

In \cite{anderson}, Anderson and Dumitrescu generalized
the concepts of the Noetherian rings and the Noetherian modules using multiplicative sets.
Authors defined a submodule $N$ of $M$ to be {\it $S$-finite} if there exist an $s \in S$ and
a finitely generated submodule $F$ of $M$ such that
$Ns \subseteq F \subseteq N$, while
an ideal $I$ of $R$ is {\it $S$-finite} if
$I$ is $S$-finite as an $R$-module.
Also, $M$ is {\it $S$-Noetherian} if
every submodule of $M$ is $S$-finite, while
$R$ is an {\it $S$-Noetherian ring} if
$R$ is $S$-Noetherian as an $R$-module.
The readers can refer to \cite{anderson, hamed1, lim, lo1, lo2, liu}
for $S$-Noetherian rings and $S$-Noetherian modules.
In \cite{kim}, Kim, Kim and Lim generalized
the concepts of SM-domains and SM-modules
using multiplicative sets.
They defined a submodule $N$ of $M$ to be {\it $S$-$w$-finite} if
there exist an $s \in S$ and a finitely generated submodule $F$ of $M$ such that
$Ns \subseteq F_w \subseteq N_w$,
while an ideal $I$ of $D$ is {\it $S$-$w$-finite} if
$I$ is $S$-$w$-finite as a $D$-module.
Also, a $w$-module $M$ is an {\it $S$-strong Mori module} ($S$-SM-module) if
every $w$-submodule of $M$ is $S$-$w$-finite; and
$D$ is an {\it $S$-strong Mori domain} ($S$-SM-domain) if
$D$ is an $S$-SM-module over $D$.
The readers can refer to \cite{chang, chang1, hamed, kim, wang2, wang1}
for ($S$-)SM-domains and ($S$-)SM-modules.

Recall that
$R$ is a Noetherian ring if and only if $R[X]$ is a Noetherian ring; and
$M$ is Noetherian if and only if $M[X]$ is Noetherian
\cite[Theorem 7.5 and Chapter 7, Exercise 10]{atiyah}.
This is well known as Hilbert basis theorem.
In \cite{anderson}, Anderson and Dumitrescu proved
the Hilbert basis theorem for $S$-Noetherian rings,
which states that
if $S$ is an anti-Archimedean subset of $R$, then
$R$ is an $S$-Noetherian ring if and only if
$R[X]$ is an $S$-Noetherian ring \cite[Proposition 9]{anderson}.
Also, Chang proved the Hilbert basis theorem for
SM-domains and SM-modules in \cite{chang, chang1}; that is,
$D$ is an SM-domain if and only if
$D[X]$ is an SM-domain \cite[Theorem 2.2]{chang}; and
for $w$-module $M$,
$M$ is an SM-module over $D$ if and only if
$M[X]$ is an SM-module over $D[X]$ \cite[Theorem 2.5]{chang1}.
In \cite{kim}, the authors proved
the Hilbert basis theorem for $S$-SM-domain,
which states that if $S$ is an anti-Archimedean subset of $D$,
then $D$ is an $S$-SM-domain if and only if
$D[X]$ is an $S$-SM-domain \cite[Theorem 2.8]{kim}.
The main purpose of this paper is
to prove the Hilbert basis theorem for $S$-Noetherian modules and
$S$-SM-modules.
To summarize,
we present the following diagram.

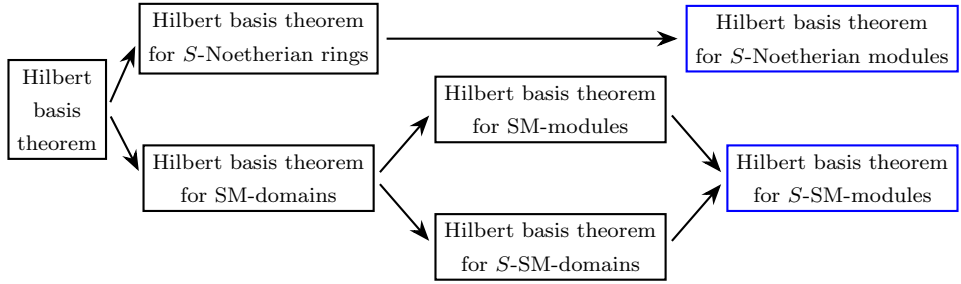
\vspace{.3cm}

\begin{center}
\begin{tikzpicture}[scale=0.1, node distance=1.8cm, thick, nodes={draw, rectangle, align=center, minimum width=1cm, minimum height=0.5cm}, arrows=-{Stealth[scale=1.2]}, shorten >=3pt, shorten <=3pt]
    \node (HBT) {\footnotesize Hilbert\\ \footnotesize basis\\ \footnotesize theorem};
    \node[right=1.5cm of HBT, draw=white] (blank) {\color{white}{\footnotesize blank}};
    \node[above=0.2cm of blank] (S-Noetherian ring) {\footnotesize Hilbert basis theorem\\\footnotesize for $S$-Noetherian rings};
    \node[below=0.2cm of blank] (SM-domain) {\footnotesize Hilbert basis theorem\\ \footnotesize for SM-domains};
    \node[right=1.77cm of SM-domain, draw=white] (blank2) {\color{white}{\footnotesize blank2}};
    \node[above=0.2cm of blank2] (SM-module) {\footnotesize Hilbert basis theorem\\ \footnotesize for SM-modules};
    \node[below=0.2cm of blank2] (S-SM-domain) {\footnotesize Hilbert basis theorem\\ \footnotesize for $S$-SM-domains};
    \node[right=1.77cm of blank2, draw=blue] (S-SM-module) {\footnotesize Hilbert basis theorem\\ \footnotesize for $S$-SM-modules};
    \node[right=4.08cm of S-Noetherian ring, draw=blue] (S-Noetherian) {\footnotesize Hilbert basis theorem\\ \footnotesize  for $S$-Noetherian modules};
    \draw (HBT.east) -- (S-Noetherian ring.west);
    \draw (HBT.east) -- (SM-domain.west);
    \draw (SM-domain.east) -- (SM-module.west);
    \draw (SM-domain.east) -- (S-SM-domain.west);
    \draw (S-SM-domain.east) -- (S-SM-module.west);
    \draw (SM-module.east) -- (S-SM-module.west);
    \draw (S-Noetherian ring.east) -- (S-Noetherian.west);
\end{tikzpicture}
\end{center}

\vspace{.2cm}

This paper consists of three sections including introduction.
In Section \ref{sec2},
we investigate some basic properties of quotient modules and
$S$-Noetherian modules.
We define a module which has finite character and
show that if $M$ is a locally $S$-Noetherian module which has finite character,
then $M$ is an $S$-Noetherian module (Proposition \ref{locally S-Noe}).
We also show that
the Hilbert basis theorem for $S$-Noetherian module
when $S$ is an anti-Archimedean subset of $R$,
{\it i.e.},
$M$ is an $S$-Noetherian $R$-module if and only if
$M[X]$ is an $S$-Noetherian $R[X]$-module if and only if
$M[X]_N$ is an $S$-Noetherian $R[X]_N$-module (Theorem \ref{main1}).
In Section \ref{sec3},
we study some properties of $w$-submodules and $S$-SM-modules.
Also, we define a module which has finite $w$-character
and then we show that
if $M$ is an $S$-SM-module,
then $M$ is a $w$-locally $S$-Noetherian module; and
if $M$ is a $w$-locally $S$-Noetherian module which has finite $w$-character,
then $M$ is an $S$-SM-module (Proposition 3.6).
Finally, we show that the Hilbert basis theorem for $S$-SM-modules
when $S$ is an anti-Archimedean subset of $D$, {\it i.e.},
$M$ is an $S$-SM-module over $D$ if and only if
$M[X]$ is an $S$-SM-module over $D[X]$ if and only if
$M[X]_{N_v}$ is an $S$-SM-module over $D[X]_{N_v}$ if and only if
$M[X]_{N_v}$ is an $S$-Noetherian $D[X]_{N_v}$-module (Theorem \ref{main2}).

To help readers better understanding this paper,
we review some definitions and notations related to star-operations.
Let $D$ be an integral domain with quotient field $K$,
${\bf F}(D)$ the set of nonzero fractional ideals of $D$ and
${\bf T}(D)$ the set of nonzero torsion-free $D$-modules.
For an $I \in {\bf F}(D)$,
set $I^{-1} := \{a \in K \,|\, aI \subseteq D\}$.
The mapping on ${\bf F}(D)$ defined by $I \mapsto I_v := (I^{-1})^{-1}$
is called the {\it $v$-operation},
and the mapping on ${\bf F}(D)$ defined by
$I \mapsto I_t :=
\bigcup \{J_v \,|\, J$ is a nonzero finitely generated fractional subideal of $I \}$
is called the {\it $t$-operation}.
An ideal $I$ of $D$ is a {\it $v$-ideal} (respectively, $t$-ideal) if
$I_v = I$ (respectively, $I_t = I$).
An ideal $J$ of $D$ is a {\it Glaz–Vasconcelos ideal} (GV-ideal),
and denoted by $J \in {\rm GV}(D)$ if
$J$ is finitely generated and $J_v = D$.
For each $M \in {\bf T}(D)$,
{\it $w$-envelop} of $M$ is the set
$M_{w_D} := \{ x \in M \otimes K \,|\, xJ \subseteq M {\rm \ for\ some\ } J \in {\rm GV}(D)\}$.
If there is no confusion,
we simply write $w$ for $w_D$.
The mapping on ${\bf T}(D)$ defined by
$M \mapsto M_w$ is called the {\it $w$-operation}.
An element $M \in {\bf T}(D)$ is a {\it $w$-module} if $M_w = M$,
while an ideal $I$ of $D$ is a {\it $w$-ideal} if
$I$ is a $w$-module as a $D$-module.
Let $*$ be the $t$-operation or the $w$-operation on $D$.
Then a proper ideal $I$ of $D$ is said to be a {\it maximal $*$-ideal} of $D$ if
there does not exist a proper $*$-ideal which properly contains $I$.
Let $*$-${\rm Max}(D)$ be the set of maximal $*$-ideals of $D$.
Then it is easy to see that if $D$ is not a field, then $*$-${\rm Max}(D) \neq \emptyset$.
The useful facts in this paper,
$t\mhyphen{\rm Max}(D) = w\mhyphen{\rm Max}(D)$ \cite[Theorem 2.16]{anderson1} and
$M_w = \bigcap_{\mathfrak{m} \in t\mhyphen{\rm Max}(D)} M_{\mathfrak{m}}$
for all nonzero $D$-modules $M$ \cite[Theorem 4.3]{anderson1}.
The readers can refer to \cite{anderson1, kang, wang} for star-operations.

\section{$S$-Noetherian modules}\label{sec2}

Let $R$ be a commutative ring with identity and
let $S$ and $T$ be multiplicative subsets of $R$.
Then $S_T=\{\frac{s}{t} \,|\, s \in S$ and $t \in T\}$ is
a multiplicative subset of $R_T$.
We start this section with simple results for a submodule of quotient modules and
a quotient module of $S$-Noetherian modules.

\begin{lemma}\label{quotient 1}
Let $R$ be a commutative ring with identity
and let $S$ and $T$ be multiplicative subsets of $R$.
Let $M$ be a unitary $R$-module.
Then the following assertions hold.
\begin{enumerate}
\item[(1)] If $A$ is an $R_T$-submodule of $M_T$,
then $A=L_T$ for some $R$-submodule $L$ of $M$.
\item[(2)] If $M$ is an $S$-Noetherian $R$-module,
then $M_T$ is an $S_T$-Noetherian $R_T$-module.
Furthermore, if $T$ consists of regular elements of $R$,
then $M_T$ is an $S$-Noetherian $R_T$-module.
\end{enumerate}
\end{lemma}

\begin{proof}
(1) Suppose that $A$ is an $R_T$-submodule of $M_T$
and let $t$ be any element of $T$.
Note that $M_T$ is an $R$-module and
the map $\varphi_t : M \to M_T$
given by $\varphi_t(m) = \frac{mt}{t}$ is an $R$-module homomorphism.
Let $L=\varphi_t^{-1}(A)$.
Then $L$ is an $R$-submodule of $M$.
Let $\frac{\ell}{v} \in L_T$, where $\ell \in L$ and $v \in  T$.
Then $\varphi_t(\ell) \in A$,
so $\frac{\ell}{v} = \frac{\ell t}{t}\frac{t}{tv}=\varphi_t(\ell)\frac{t}{tv} \in A$.
Hence $L_T \subseteq A$.
For the reverse containment,
let $\ell \in M$ and $v \in T$ with $\frac{\ell}{v}\in A$.
Then $\varphi_t(\ell)=\frac{\ell t}{t} = \frac{\ell}{v}\frac{tv}{t} \in A$,
so $\ell \in \varphi_t^{-1}(A)=L$.
This implies that $\frac{\ell}{v} \in L_T$.
Hence $A \subseteq L_T$,
and thus $A=L_T$.

(2) Let $A$ be an $R_T$-submodule of $M_T$.
Then by (1), $A=L_T$ for some $R$-submodule $L$ of $M$.
Since $M$ is an $S$-Noetherian $R$-module,
there exist $s \in S$ and
$\ell_1, \dots, \ell_n \in L$ such that
\begin{center}
$Ls \subseteq \ell_1 R + \cdots + \ell_n R$.
\end{center}
Fix an element $t \in T$.
Note that $(Ls)_T = L_T \frac{s}{t}$ and
$(\ell_kR)_T = \frac{\ell_k}{t}R_T$ for all $1 \leq k \leq n$, so we have
\begin{center}
$A\frac{s}{t} = L_T\frac{s}{t} = (Ls)_T
\subseteq (\ell_1R + \cdots + \ell_nR)_T
= \frac{\ell_1}{t} R_T  + \cdots + \frac{\ell_n}{t} R_T
\subseteq L_T = A$.
\end{center}
Hence $A$ is an $S_T$-finite $R_T$-submodule of $M_T$,
which means that $M_T$ is an $S_T$-Noetherian $R_T$-module.

Note that if $T$ consists of regular elements of $R$,
then $R$ can be naturally embedded in $R_T$.
Hence we may assume that $S$ is a multiplicative subset of $R_T$.
Thus the second argument holds.
\end{proof}

Let $R$ be a commutative ring with identity and
let $P$ be a prime ideal of $R$.
Then $S:=R \setminus P$ is a (saturated) multiplicative subset of $R$.
Let $M$ be a unitary $R$-module.
We say that $M$ is {\em $P$-finite} if
$M$ is $S$-finite; and
$M$ is a {\em $P$-Noetherian module} if
$M$ is an $S$-Noetherian module.
For an element $r \in R$ and an $R$-submodule $L$ of $M$,
we set $L:r=\{x \in M \,|\, xr \in L\}$. 
It is easy to see that $L:r$ is an $R$-submodule of $M$ containing $L$.

\begin{proposition}
Let $R$ be a commutative ring with identity,
$\mathfrak{m}$ a maximal ideal of $R$ and
$M$ a torsion-free unitary $R$-module.
Then the following assertions are equivalent.
\begin{enumerate}
\item[(1)]
$M$ is an $\mathfrak{m}$-Noetherian module.
\item[(2)]
$M_{\mathfrak{m}}$ is a Noetherian $R_{\mathfrak{m}}$-module and
every nonzero finitely generated $R$-submodule $L$ of $M$,
there exists an element $s \in R \setminus \mathfrak{m}$ such that
$L_{\mathfrak{m}}\cap M = L:s$.
\end{enumerate}
\end{proposition}

\begin{proof}
(1) $\Rightarrow$ (2)
Let $A$ be a nonzero $R_{\mathfrak{m}}$-submodule of $M_{\mathfrak{m}}$.
Then by Lemma \ref{quotient 1}(1),
$A = B_{\mathfrak{m}}$
for some $R$-submodule $B$ of $M$.
Since $M$ is an $\mathfrak{m}$-Noetherian module,
there exist $s \in R \setminus \mathfrak{m}$ and
$b_1, \dots, b_n \in B$ such that
$Bs \subseteq b_1R + \cdots + b_nR$, so we obtain
\begin{eqnarray*}
B_{\mathfrak{m}} &=& (Bs)_{\mathfrak{m}}\\
&\subseteq& b_1R_{\mathfrak{m}} + \cdots + b_nR_{\mathfrak{m}}\\
&\subseteq & B_{\mathfrak{m}}.
\end{eqnarray*}
Hence $B_{\mathfrak{m}} = b_1R_{\mathfrak{m}} + \cdots + b_nR_{\mathfrak{m}}$.
Thus $M_{\mathfrak{m}}$ is a Noetherian $R_{\mathfrak{m}}$-module.
For the second argument,
let $L$ be a nonzero finitely generated $R$-submodule of $M$.
Since $L_{\mathfrak{m}} \cap M$ is an $R$-submodule of $M$,
there exist $u \in R \setminus \mathfrak{m}$ and
$c_1, \dots, c_m \in L_{\mathfrak{m}} \cap M$ such that
\begin{center}
$(L_{\mathfrak{m}} \cap M)u \subseteq  c_1R + \cdots + c_m R$.
\end{center}
For each $i=1, \dots, m$, take an element $t_i \in R\setminus \mathfrak{m}$
such that $c_it_i \in L$.
Let $t = t_1\cdots t_m$ and let $s=tu$.
Then $(c_1R + \cdots + c_mR)t \subseteq L$.
Hence we obtain
\begin{center}
$(L_{\mathfrak{m}} \cap M)s \subseteq (c_1R + \cdots + c_mR)t \subseteq L$.
\end{center}
Thus $L_{\mathfrak{m}} \cap M=L:s$.

(2) $\Rightarrow$ (1)
Let $L$ be a nonzero $R$-submodule of $M$.
Then $L_{\mathfrak{m}}$ is an $R_{\mathfrak{m}}$-submodule of $M_{\mathfrak{m}}$.
Since $M_{\mathfrak{m}}$ is a Noetherian $R_{\mathfrak{m}}$-module,
$L_{\mathfrak{m}} = a_1R_{\mathfrak{m}} + \cdots + a_nR_{\mathfrak{m}}$
for some $a_1, \dots, a_n \in L$.
Therefore by the assumption, we have
\begin{eqnarray*}
L &\subseteq & L_{\mathfrak{m}}\cap M \\
&=& (a_1R_{\mathfrak{m}} + \cdots + a_nR_{\mathfrak{m}}) \cap M\\
&=& (a_1R + \cdots + a_nR):s
\end{eqnarray*}
for some $s \in R \setminus \mathfrak{m}$.
Hence $Ls \subseteq a_1R + \cdots + a_nR$,
which means that $L$ is $\mathfrak{m}$-finite.
Thus $M$ is an $\mathfrak{m}$-Noetherian module.
\end{proof}

\begin{proposition}
Let $R$ be a commutative ring with identity and
let $M$ be a torsion-free unitary $R$-module.
Then the following conditions are equivalent.
\begin{enumerate}
\item[(1)]
$M$ is a Noetherian module.
\item[(2)]
$M$ is a $P$-Noetherian module for all $P \in {\rm Spec}(R)$.
\item[(3)]
$M$ is an $\mathfrak{m}$-Noetherian module for all $\mathfrak{m} \in {\rm Max}(R)$.
\end{enumerate}
\end{proposition}

\begin{proof}
(1) $\Rightarrow$ (2) $\Rightarrow$ (3) These implications are obvious.

(3) $\Rightarrow$ (1) Suppose that $M$ is an $\mathfrak{m}$-Noetherian module
for all $\mathfrak{m} \in {\rm Max}(R)$ and
let $L$ be an $R$-submodule of $M$.
Then for each $\mathfrak{m} \in {\rm Max}(R)$,
there exist an element $s_{\mathfrak{m}} \in R \setminus \mathfrak{m}$ and
a finitely generated $R$-submodule $F_{\mathfrak{m}}$ of $L$ such that
$Ls_{\mathfrak{m}} \subseteq F_{\mathfrak{m}}$.
Let $S = \{s_{\mathfrak{m}} \,|\, \mathfrak{m} \in {\rm Max}(R)\}$.
Then $S$ is not contained in any maximal ideal of $R$, so
there exist $s_{\mathfrak{m}_1}, \dots, s_{\mathfrak{m}_n} \in S$ such that
$(s_{\mathfrak{m}_1}, \dots, s_{\mathfrak{m}_n}) = R$.
Therefore we obtain
\begin{eqnarray*}
L &=& L(s_{\mathfrak{m}_1}, \dots, s_{\mathfrak{m}_n})\\
&\subseteq & F_{\mathfrak{m}_1} + \cdots + F_{\mathfrak{m}_n}\\
&\subseteq & L.
\end{eqnarray*}
Hence $L = F_{\mathfrak{m}_1} + \cdots + F_{\mathfrak{m}_n}$.
Note that $F_{\mathfrak{m}_1} + \cdots + F_{\mathfrak{m}_n}$ is finitely generated.
Thus $M$ is a Noetherian module.
\end{proof}

Let $D$ be an integral domain,
$S$ a multiplicative subset of $D$ and
$M$ a unitary $D$-module.
We define $M$ to be {\it locally $S$-Noetherian} if
for each maximal ideal $\mathfrak{m}$ of $D$,
$M_{\mathfrak{m}}$ is an $S$-Noetherian $D_{\mathfrak{m}}$-module.
Let $L$ be a $D$-submodule of $M$.
Then it is easy to see that
$(L:M) = \{d \in D \,|\, Md \subseteq L \}$ is an ideal of $D$.
Recall that $D$ has {\it finite character} if
every nonzero nonunit in $D$ belongs to
only finitely many maximal ideals of $D$
(equivalently, each nonzero proper ideal of $D$ is contained in
only finitely many maximal ideals of $D$).
This concept can be generalized to the module version as follows:
$M$ has {\it finite character} if
for each nonzero element $a$ of $M$ with $(aD : M) \neq D$,
$(aD : M)$ is contained in only finitely many maximal ideals of $D$.
It is easy to show that
$M$ has finite character if and only if
for each nonzero proper $D$-submodule $L$ of $M$,
$(L:M)$ is contained in only finitely many maximal ideals of $D$.

\begin{proposition}\label{locally S-Noe}
Let $D$ be an integral domain,
$S$ a multiplicative subset of $D$
and $M$ a torsion-free unitary $D$-module.
Then the following assertions hold.
\begin{enumerate}
\item[(1)]
If $M$ is an $S$-Noetherian module,
then $M$ is a locally $S$-Noetherian module.
\item[(2)]
If $M$ is a locally $S$-Noetherian module
which has finite character,
then $M$ is an $S$-Noetherian module.
\end{enumerate}
\end{proposition}

\begin{proof}
(1) This is an immediate consequence of Lemma \ref{quotient 1}(2).

(2) Let $A$ be a $D$-submodule of $M$ and
let $a$ be a nonzero element of $A$ such that $(aD:M) \neq D$.
Since $M$ has finite character,
$(aD:M)$ is contained in only finitely many maximal ideals of $D$,
say $\mathfrak{m}_1,\dots, \mathfrak{m}_n$.
Since $M_{\mathfrak{m}_1}, \dots, M_{\mathfrak{m}_n}$ are $S$-Noetherian,
for each $i \in \{1, \dots, n\}$,
there exist an element $s_i \in S$ and
a finitely generated $D$-submodule $F_i$ of $A$ such that
$A_{\mathfrak{m}_i}s_i \subseteq ({F_i})_{\mathfrak{m}_i}$.
Let $s = s_1 \cdots s_n$ and let $F = aD + F_1 + \cdots + F_n$.
Then $A_{\mathfrak{m}_i}s \subseteq F_{\mathfrak{m}_i}$
for all $i = 1, \dots, n$.
Let $\mathfrak{m}'$ be a maximal ideal of $D$
which is distinct from $\mathfrak{m}_1, \dots, \mathfrak{m}_n$.
Then $(aD:M) \nsubseteq \mathfrak{m}'$, so
we can pick an element $r \in (aD:M) \setminus \mathfrak{m}'$.
Therefore $\frac{m}{1} = \frac{mr}{r} \in (aD)_{\mathfrak{m}'}$ for all $m \in M$.
This shows that $(aD)_{\mathfrak{m}'} = M_{\mathfrak{m}'}$, which indicates that
$A_{\mathfrak{m}'} = M_{\mathfrak{m}'} = F_{\mathfrak{m}'}$.
Hence $A_{\mathfrak{m}}s \subseteq F_{\mathfrak{m}}$
for all maximal ideals $\mathfrak{m}$ of $D$. 
Consequently, we have
\begin{eqnarray*}
As &=& \left(\bigcap_{\mathfrak{m} \in {\rm Max}(D)} A_{\mathfrak{m}}\right)s\\
&\subseteq & \bigcap_{\mathfrak{m} \in {\rm Max}(D)} A_{\mathfrak{m}}s\\
&\subseteq & \bigcap_{\mathfrak{m} \in {\rm Max}(D)} F_{\mathfrak{m}}\\
&=& F,
\end{eqnarray*}
where the equalities follow from \cite[Theorem 4.3]{anderson1}.
Since $F$ is a finitely generated $D$-submodule of $A$,
$A$ is $S$-finite.
Thus $M$ is an $S$-Noetherian $D$-module.
\end{proof}

The next example shows that
the converse of Proposition \ref{locally S-Noe}(1) does not generally hold.

\begin{example}
{\rm
Let $\mathbb{Z}_2$ be the ring of integers modulo $2$
and let $R = \prod_{i \in \mathbb{N}} \mathbb{Z}_2$.

(1) Note that $\mathbb{Z}_2 \times \{0\} \times \{0\} \times \cdots \subsetneq
\mathbb{Z}_2 \times \mathbb{Z}_2 \times \{0\} \times \{0\} \times \cdots \subsetneq \cdots$ is
a strict ascending chain of ideals of $R$, so
$R$ is not a Noetherian ring.

(2) Note that ${\rm Max}(R)=\{\prod_{i \in \mathbb{N}} A_i \,|\,$for each
$j \in \mathbb{N}$, $A_j=\{0\}$ and $A_i=\mathbb{Z}_2$ for all $i \neq j\}$,
so for all $M \in {\rm Max}(R)$,
$R_M$ has only two elements.
Hence $R$ is a locally Noetherian ring.
}
\end{example}

Let $R$ be a commutative ring with identity and
let $M$ be a unitary $R$-module.
For an element $f \in M[X]$,
the {\em content module} $c(f)$ of $f$ is defined to be
the $R$-submodule of $M$ generated by
the coefficients of $f$.
In particular, if $M=R$, then $c(f)$ is called
the {\em content ideal} of $R$.
Let $N = \{f \in R[X] \,|\, c(f) = R\}$.
Then $N$ is a (saturated) regular multiplicative subset of $R[X]$ \cite[page 17]{n} (or \cite[page 559]{a}).
The quotient module $M[X]_N$ of $M[X]$ by $N$ is usually called
the {\it Nagata module} of $M$.
Recall that a multiplicative subset $S$ of $R$ is
{\it anti-Archimedean} if
$\bigcap_{n\geq 1}s^n R \cap S \neq \emptyset$.
Now, we give the main result in this section
which involves the Hilbert basis theorem and
the Nagata module extension for $S$-Noetherian modules.

\begin{theorem} \label{main1}
Let $R$ be a commutative ring with identity,
$S$ an anti-Archimedean subset of $R$
and $M$ a unitary $R$-module.
Then the following statements are equivalent.
\begin{enumerate}
\item[(1)] $M$ is an $S$-Noetherian $R$-module.
\item[(2)] $M[X]$ is an $S$-Noetherian $R[X]$-module.
\item[(3)] $M[X]_N$ is an $S$-Noetherian $R[X]_N$-module.
\end{enumerate}
\end{theorem}

\begin{proof}
(1) $\Rightarrow$ (2)
Let $A$ be an $R[X]$-submodule of $M[X]$.
For each $k \geq 0$, let $B_k$ be the set consisting of zero and
the leading coefficients of the polynomials in $A$ of
degree less than or equal to $k$ and
let $B = \bigcup_{k \geq 0} B_k$.
Then each $B_k$ and $B$ are $R$-submodules of $M$
such that $B_k \subseteq B_{k+1}$ for all $k \geq 0$.
Since $M$ is an $S$-Noetherian $R$-module,
there exist $t \in S$ and $b_1, \dots, b_n \in B$ such that
$Bt \subseteq b_1R + \cdots + b_n R$.
Take an integer $d$ so that $b_1, \dots, b_n \in B_d$.
Then $B t \subseteq b_1R + \cdots + b_n R \subseteq B_d$.
Again, since $M$ is an $S$-Noetherian $R$-module,
for each $j \in \{0, \dots, d\}$,
there exist $s_j \in S$ and $b_{j1},\dots, b_{jk_j} \in B_j$ such that
$B_j s_j \subseteq b_{j1}R + \cdots + b_{jk_j}R$.
Let $s=s_0\cdots s_d t$.
Then we obtain
\begin{center}
$B s \subseteq b_1R + \cdots + b_nR \subseteq B_d$
\end{center}
and for all $j \in \{0, \dots, d\}$,
\begin{center}
$B_j s \subseteq b_{j1}R + \cdots + b_{jk_j}R$.
\end{center}
For each $j \in \{0,\dots, d\}$ and $\ell \in \{1, \dots, k_j\}$,
let $f_{j\ell} = b_{j\ell}X^j +$(lower terms) $\in A$.

Now, we claim that $Au \subseteq \sum_{0\leq i \leq d}\sum_{1\leq j \leq k_i} f_{ij}R[X]$
for some $u \in S$.
Let $f = aX^m +$(lower terms) $\in A$.
First, we suppose that $m \geq d+1$.
Then $a \in B$, so $as \in B_d$,
which implies that $as^2 \in b_{d1}R + \cdots + b_{dk_d}R$.
Therefore $as^2 = b_{d1}r_1 + \cdots + b_{dk_d}r_{k_d}$
for some $r_1, \dots, r_{k_d} \in R$.
Let $\alpha = fs^2 - \sum_{\ell = 1}^{k_d}f_{d\ell}r_{\ell}X^{m-d}$.
Then $\alpha \in A$ with $\deg(\alpha) \leq m-1$.
By repeating this process,
we have $q_1 \in \mathbb{N}$ and $g_1, \dots, g_{k_d} \in R[X]$ such that
$\beta := fs^{q_1} - \sum_{\ell =1}^{k_d} f_{d\ell}g_{\ell} \in A$ and
$\deg(\beta) \leq d$.
Since the leading coefficient of $\beta$ belongs to $B_{\deg(\beta)}$,
there exist $r'_1,\dots, r'_{k_{\deg(\beta)}} \in R$ such that
$\gamma := \beta s - \sum_{\ell = 1}^{k_{\deg(\beta)}} f_{\deg(\beta)\ell}r'_{\ell} \in A$ and
$\deg(\gamma) \leq \deg(\beta)-1$.
If we still have $\gamma \neq 0$, then we repeat the same process.
After finitely many steps,
we obtain
\begin{center}
$fs^{q_2} \in \sum_{0\leq i \leq d}\sum_{1\leq j \leq k_i} f_{ij}R[X]$
\end{center}
for some $q_2 \in \mathbb{N}$.
Second, we suppose that $m\leq d$.
Then a similar argument as in the previous case shows that
\begin{center}
$f s^{q_3} \in 
\sum_{0\leq i \leq d}\sum_{1\leq j \leq k_i} f_{ij}R[X]$
\end{center}
for some $q_3 \in \mathbb{N}$.
Since $S$ is an anti-Archimedean subset of $R$,
there exists an element $u \in \bigcap_{n\geq 1}s^n R \cap S$, so we have
\begin{center}
$fu \in \sum_{0\leq i \leq d}\sum_{1\leq j \leq k_i} f_{ij}R[X]$.
\end{center}
Since $f$ was arbitrarily chosen in $A$, we obtain
\begin{center}
 $A u \subseteq \sum_{0\leq i \leq d}\sum_{1\leq j \leq k_i} f_{ij}R[X]$.
\end{center}
Hence $A$ is an $S$-finite $R[X]$-submodule of $M[X]$.
Thus $M[X]$ is an $S$-Noetherian $R[X]$-module.

(2) $\Rightarrow$ (3)
This implication follows directly from Lemma \ref{quotient 1}(2).

(3) $\Rightarrow$ (1)
Let $A$ be an $R$-submodule of $M$.
Then $A[X]_N$ is an $R[X]_N$-submodule of $M[X]_N$.
Since $M[X]_N$ is an $S$-Noetherian $R[X]_N$-module,
there exist $s \in S$, $f_1, \dots, f_n \in A[X]$
and $g_1, \dots, g_n \in N$ such that
\begin{center}
$A[X]_N s \subseteq  \frac{f_1}{g_1} R[X]_N  + \cdots + \frac{f_n}{g_n} R[X]_N $.
\end{center}
Let $a \in A$.
Then we can find $h_1, \dots, h_n \in R[X]$
and $\alpha_1, \dots, \alpha_n \in N$ such that
$as =
\frac{f_1}{g_1} \frac{h_1}{\alpha_1}+\cdots +\frac{f_n}{g_n} \frac{h_n}{\alpha_n}$.
Let $\alpha = \prod_{i = 1}^{n} g_i\alpha_i$ and
for each $i=1, \dots, n$, let $\beta_i = \frac{\alpha h_i}{g_i \alpha_i}$.
Then $as =
\frac{f_1\beta_1+ \cdots + f_n\beta_n}{\alpha}$, so
we have
\begin{eqnarray*}
as\alpha &=&f_1\beta_1 + \cdots + f_n\beta_n\\
&\in& (c(f_1) + \cdots + c(f_n))[X].
\end{eqnarray*}
Since $\alpha \in N$, $as \in c(f_1) + \cdots + c(f_n)$.
Therefore $As \subseteq c(f_1) + \cdots + c(f_n)$.
Note that $c(f_1) + \cdots + c(f_n)$ is a finitely generated $R$-submodule of $A$.
Hence $A$ is an $S$-finite $R$-submodule of $M$.
Thus $M$ is an $S$-Noetherian $R$-module.
\end{proof}

\section{$S$-strong Mori modules}\label{sec3}

We start this section with some observations for $S$-$w$-finite modules.

\begin{remark}\label{N_w -> N}
{\rm
Let $D$ be an integral domain,
$S$ a multiplicative subset of $D$ and
$M$ a torsion-free $w$-module as a $D$-module.

(1) Let $L$ be a nonzero $D$-submodule of $M$.
Then $L_w$ is a $w$-submodule of $M$.
If $L_w$ is $S$-$w$-finite,
then there exist an element $s \in S$ and
a $w$-finite type submodule $F$ of $L_w$
such that $L_w s \subseteq F$, so $Ls \subseteq F$.
Conversely, if there exist an element $s \in S$
and a $w$-finite type submodule $F$ of $L_w$
such that $Ls \subseteq F$,
then $L_w s \subseteq F$.
Hence we may extend the concept of $S$-$w$-finite modules to
any nonzero submodule of a $w$-module as follows:
A nonzero submodule $L$ of $M$ is {\em $S$-$w$-finite} if
there exist an element $s \in S$ and
a $w$-finite type submodule $F$ of $L_w$
such that $Ls \subseteq F$.

(2) By (1), $M$ is an $S$-SM-module
if and only if every nonzero submodule of $M$ is $S$-$w$-finite.

(3) Suppose that $L$ is an $S$-$w$-finite submodule of $M$.
Then we can find $s \in S$ and $a_1, \dots, a_n \in L_w$ such that
$Ls \subseteq (a_1 D + \cdots + a_n D)_w$, so
for each $i=1, \dots, n$,
there exists an element $J_i \in {\rm GV}(D)$ such that
$a_iJ_i \subseteq L$.
Let $J=J_1 \cdots J_n$.
Then $J \in {\rm GV}(D)$ \cite[Lemma 1.1]{wang1} (or \cite[Lemma 2.3(3)]{hl})
and $a_iJ \subseteq L$ for all $i \in \{1, \dots, n\}$,
so we obtain
\begin{eqnarray*}
(a_1D + \cdots + a_nD)_w &=& ((a_1D + \cdots + a_nD)J)_w\\
&=& (a_1J+\cdots+a_nJ)_w,
\end{eqnarray*}
where the first equality follows from \cite[Proposition 2.7]{wang1}.
Hence $Ls \subseteq (a_1J + \cdots + a_nJ)_w$.
Note that for all $i \in \{1, \dots, n\}$,
$a_iJ$ is a finitely generated submodule of $L$.
Thus we may assume that $a_1, \dots, a_n \in L$
by replacing $a_1D+\cdots+a_nD$ by $a_1J+\cdots+a_nJ$.
}
\end{remark}

\begin{lemma}\label{basic lem for S-SM module}
Let $D$ be an integral domain and let $S$ be a multiplicative subset of $D$.
Let $M$ be a torsion-free $D$-module,
$w$ the $w$-operation on $M$ and
$\overline{w}$ the $w$-operation on $M_S$ as a $D_S$-module.
Suppose that $M$ is a $w$-module.
If $L$ is a nonzero $D$-submodule of $M$, then the following assertions hold.
\begin{enumerate}
\item[(1)] If $L$ is a $w$-submodule of $M$,
then $L_S \cap M$ is a $w$-submodule of $M$.
\item[(2)] If $L_S$ is a $\overline{w}$-submodule of $M_S$,
then $L_S \cap M$ is a $w$-submodule of $M$.
\item[(3)] $(L_w)_S \subseteq (L_S)_{\overline{w}}$ and
$((L_{w})_S)_{\overline{w}} = (L_S)_{\overline{w}}$.
\end{enumerate}
\end{lemma}

\begin{proof}
(1) Let $x \in (L_S \cap M)_w$.
Then $xJ \subseteq L_S \cap M$ for some $J \in {\rm GV}(D)$.
Since $J$ is finitely generated,
$xsJ \subseteq L$ for some $s \in S$,
so $xs \in L_w = L$.
Also, $x \in M_w=M$.
Hence $x \in L_S \cap M$.
Thus $L_S \cap M$ is a $w$-submodule of $M$.

(2) Let $x \in (L_S \cap M)_{w}$.
Then $xJ \subseteq L_S \cap M$ for some $J \in {\rm GV}(D)$, so
$xJD_S \subseteq L_S$.
Note that $JD_S \in {\rm GV}(D_S)$ (cf. \cite[Lemma 3.4(1)]{kang}), so
$x \in (L_S)_{\overline{w}} = L_S$.
Also, $x \in M_w=M$.
Hence $x \in L_S \cap M$.
Thus $L_S \cap M$ is a $w$-submodule of $M$.

(3) Let $x \in (L_{w})_S$.
Then $xs \in L_{w}$ for some $s \in S$,
so there exists an element $J \in {\rm GV}(D)$ such that $xsJ \subseteq L$.
Since $xJD_S \subseteq L_S$ and $JD_S \in {\rm GV}(D_S)$,
$x \in (L_S)_{\overline{w}}$.
Hence $(L_w)_S \subseteq (L_S)_{\overline{w}}$.
Also, by the previous inclusion,
$((L_{w})_S)_{\overline{w}} \subseteq ((L_S)_{\overline{w}})_{\overline{w}}
= (L_S)_{\overline{w}}$.
Thus $((L_{w})_S)_{\overline{w}} = (L_S)_{\overline{w}}$.
\end{proof}

Let $D$ be an integral domain and let $P$ be a prime ideal of $D$.
Then $S:=D \setminus P$ is a (saturated) multiplicative subset of $D$.
Let $M$ be a $w$-module and let $L$ be a nonzero submodule of $M$.
We say that $L$ is {\em $P$-$w$-finite} if
$L$ is $S$-$w$-finite; and
$M$ is a {\em $P$-strong Mori module} ($P$-SM-module) if
$M$ is an $S$-SM-module.

\begin{proposition}
Let $D$ be an integral domain, $\mathfrak{m}$ a maximal $w$-ideal of $D$
and $M$ a torsion-free $w$-module as a $D$-module.
Then the following assertions are equivalent.
\begin{enumerate}
\item[(1)]
$M$ is an $\mathfrak{m}$-SM-module.
\item[(2)]
$M_{\mathfrak{m}}$ is a Noetherian $D_{\mathfrak{m}}$-module and
for every nonzero finitely generated $D$-submodule $L$ of $M$,
there exists an element $s \in D \setminus \mathfrak{m}$ such that
$(L_w)_{\mathfrak{m}} \cap M = L_w : s$.
\end{enumerate}
\end{proposition}

\begin{proof}
(1) $\Rightarrow$ (2)
Let $A$ be a nonzero $D_{\mathfrak{m}}$-submodule of $M_{\mathfrak{m}}$.
Then by Lemma \ref{quotient 1}(1),
$A = B_{\mathfrak{m}}$
for some $D$-submodule $B$ of $M$.
Since $M$ is an $\mathfrak{m}$-SM-module,
there exist $s \in D \setminus \mathfrak{m}$ and
$b_1, \dots, b_n \in B$ such that
$Bs \subseteq (b_1D + \cdots + b_nD)_w$, so we obtain
\begin{eqnarray*}
B_{\mathfrak{m}} &=& (Bs)_{\mathfrak{m}}\\
&\subseteq & ((b_1D + \cdots + b_nD)_w)_{\mathfrak{m}}\\
&=& b_1D_{\mathfrak{m}} + \cdots + b_nD_{\mathfrak{m}}\\
&\subseteq & B_{\mathfrak{m}},
\end{eqnarray*}
where the second equality comes from
\cite[Remark before Proposition 4.6]{wang1}
(or \cite[Theorem 4.3]{anderson1})
since $t$-${\rm Max}(D) = w$-${\rm Max}(D)$ \cite[Theorem 2.16]{anderson1}.
Hence $B_{\mathfrak{m}} = b_1D_{\mathfrak{m}} + \cdots + b_nD_{\mathfrak{m}}$.
Therefore $M_{\mathfrak{m}}$ is a Noetherian $D_{\mathfrak{m}}$-module.
For the remaining argument,
let $L$ be a nonzero finitely generated $D$-submodule of $M$.
Then $(L_w)_{\mathfrak{m}} \cap M$ is a $w$-submodule of $M$
by Lemma \ref{basic lem for S-SM module}(1), so
there exist $t \in D \setminus \mathfrak{m}$ and
$c_1, \dots, c_m \in (L_w)_{\mathfrak{m}} \cap M$ such that
\begin{center}
$((L_w)_{\mathfrak{m}} \cap M)t \subseteq  (c_1D + \cdots + c_m D)_w$
\end{center}
and $c_1t_1, \dots, c_mt_m \in L_w$
for some $t_1, \dots, t_m \in D\setminus \mathfrak{m}$.
Let $t' = t_1\cdots t_m$.
Then $(c_1D + \cdots + c_mD)_w t'\subseteq L_w$.
Therefore
\begin{center}
$((L_w)_{\mathfrak{m}} \cap M)tt'
\subseteq (c_1D + \cdots + c_mD)_wt'
\subseteq L_w$.
\end{center}
This fact implies that
$(L_w)_{\mathfrak{m}} \cap M = L_w : s$,
where $s = tt'$.

(2) $\Rightarrow$ (1)
Let $L$ be a nonzero $D$-submodule of $M$.
Then $L_{\mathfrak{m}}$ is a $D_{\mathfrak{m}}$-submodule of $M_{\mathfrak{m}}$, so
$L_{\mathfrak{m}} = a_1D_{\mathfrak{m}} + \cdots + a_nD_{\mathfrak{m}}$
for some $a_1, \dots, a_n \in L$.
Hence
\begin{eqnarray*}
L &\subseteq & L_{\mathfrak{m}}\cap M \\
&=& (a_1D_{\mathfrak{m}} + \cdots + a_nD_{\mathfrak{m}}) \cap M\\
&=& ((a_1D + \cdots + a_nD)_w)_{\mathfrak{m}} \cap M\\
&=& (a_1D + \cdots + a_nD)_w : s
\end{eqnarray*}
for some $s \in D\setminus \mathfrak{m}$,
where the third equality comes from \cite[Remark before Proposition 4.6]{wang1},
which means that $Ls \subseteq (a_1D + \cdots + a_nD)_w$.
Thus $L$ is $\mathfrak{m}$-$w$-finite.
Consequently, $M$ is an $\mathfrak{m}$-SM-module.
\end{proof}

\begin{proposition}
Let $D$ be an integral domain and 
let $M$ be a torsion-free $w$-module as a $D$-module.
Then the following conditions are equivalent.
\begin{enumerate}
\item[(1)] $M$ is an SM-module.
\item[(2)] $M$ is a $P$-SM-module for all $P \in w$-${\rm Spec}(D)$.
\item[(3)] $M$ is an $\mathfrak{m}$-SM-module for all $\mathfrak{m} \in w$-${\rm Max}(D)$.
\end{enumerate}
\end{proposition}

\begin{proof}
(1) $\Rightarrow$ (2) $\Rightarrow$ (3) These implications are obvious.

(3) $\Rightarrow$ (1) Suppose that $M$ is an $\mathfrak{m}$-SM-module
for all $\mathfrak{m} \in w$-${\rm Max}(D)$ and
let $L$ be a $w$-submodule of $M$.
Then for each $\mathfrak{m} \in w$-Max$(D)$,
there exist an element $s_{\mathfrak{m}} \in D \setminus \mathfrak{m}$ and
a finitely generated $D$-submodule $F_{\mathfrak{m}}$ of $L$ such that
$Ls_{\mathfrak{m}} \subseteq (F_{\mathfrak{m}})_w$.
Let $S = \{s_{\mathfrak{m}} \,|\, \mathfrak{m} \in w$-${\rm Max}(D)\}$.
Then $S$ is not contained in any maximal $w$-ideal of $D$, so
there exist $s_{\mathfrak{m}_1}, \dots, s_{\mathfrak{m}_n} \in S$ such that
$(s_{\mathfrak{m}_1}, \dots, s_{\mathfrak{m}_n})_w = D$.
Hence we obtain
\begin{eqnarray*}
L &=& (L(s_{\mathfrak{m}_1}, \dots, s_{\mathfrak{m}_n})_w)_w\\
&=& (L(s_{\mathfrak{m}_1}, \dots, s_{\mathfrak{m}_n}))_w\\
&\subseteq & (F_{\mathfrak{m}_1} + \cdots + F_{\mathfrak{m}_n})_w\\
&\subseteq & L.
\end{eqnarray*}
Thus $L = (F_{\mathfrak{m}_1} + \cdots + F_{\mathfrak{m}_n})_w$.
Consequently, $M$ is an SM-module.
\end{proof}

Recall that an integral domain $D$
has {\it finite $w$-character} if
for each nonzero nonunit in $D$ belongs to
only finitely many maximal $w$-ideals of $D$,
or equivalently, for each nonzero proper ideal of $D$ is contained in
only finitely many maximal $w$-ideals of $D$.
Generalizing this, a finite $w$-character can be defined in the module as follows:
A $D$-module $M$ has {\it finite $w$-character} if
for each nonzero element $a$ of $M$ with $(aD : M) \neq D$,
$(aD : M)$ is contained in only finitely many maximal $w$-ideals of $D$.
It is easy to show that
$M$ has finite $w$-character if and only if
for each nonzero proper $D$-submodule $L$ of $M$,
$(L:M)$ is contained in only finitely many maximal $w$-ideals of $D$.
Also, it is easy to show that
every commutative ring with identity which has finite $w$-character has
finite $w$-character as module.
Recall that a $D$-module $M$ is a {\it $w$-locally $S$-Noetherian $D$-module} if
for each maximal $w$-ideal $\mathfrak{m}$,
$M_{\mathfrak{m}}$ is an $S$-Noetherian $D_{\mathfrak{m}}$-module.

\begin{proposition}\label{w-locally Noe}
Let $D$ be an integral domain, $S$ a multiplicative subset of $D$
and $M$ a torsion-free $w$-module as a $D$-module.
Then the following assertions hold.
\begin{enumerate}
\item[(1)]
If $M$ is an $S$-SM-module,
then $M$ is a $w$-locally $S$-Noetherian module.
\item[(2)]
If $M$ is a $w$-locally $S$-Noetherian module which
has finite $w$-character, then
$M$ is an $S$-SM-module.
\end{enumerate}
\end{proposition}

\begin{proof}
(1) Let $\mathfrak{m}$ be a maximal $w$-ideal of $D$ and
let $A$ be a $D_{\mathfrak{m}}$-submodule of $M_{\mathfrak{m}}$.
Then by Lemma \ref{quotient 1}(1), $A = B_{\mathfrak{m}}$
for some $D$-submodule $B$ of $M$, so
there exist $s \in S$ and $b_1, \dots, b_n \in B$ such that 
$A's \subseteq (b_1 D + \cdots + b_nD)_w$.
Therefore we obtain
\begin{center}
$As = B_{\mathfrak{m}}s \subseteq ((b_1 D + \cdots + b_nD)_w)_{\mathfrak{m}}
= b_1D_{\mathfrak{m}} + \cdots + b_nD_{\mathfrak{m}}$,
\end{center}
where the last equality comes from \cite[Remark before Proposition 4.6]{wang1}.
Hence $A$ is $S$-finite.
Thus $M_{\mathfrak{m}}$ is an $S$-Noetherian $D_{\mathfrak{m}}$-module
for each $\mathfrak{m} \in w$-Max$(D)$.
Consequently, $M$ is a $w$-locally $S$-Noetherian module.

(2) Suppose that $M$ is a $w$-locally $S$-Noetherian module
which has finite $w$-character and
let $A$ be a $D$-submodule of $M$.
Let $a$ be a nonzero element of $A$ such that
$(aD : M) \neq D$.
Then $(aD:M)$ is contained in only finitely many maximal $w$-ideals of $D$,
say $\mathfrak{m}_1, \dots, \mathfrak{m}_m$.
Since for each $i = 1, \dots, m$, $M_{\mathfrak{m}_i}$ is $S$-Noetherian,
we obtain that there exist $s_i \in S$ and
a finitely generated $D$-submodule $F_i$ of $A$ such that
$A_{\mathfrak{m}_i}s_i \subseteq (F_i)_{\mathfrak{m}_i}$.
Let $s = s_1\cdots s_m$ and $F = aD + F_1 + \cdots + F_m$.
Then $A_{\mathfrak{m}_i}s \subseteq F_{\mathfrak{m}_i}$ for all $i = 1, \dots, m$.
Let $\mathfrak{m}' \neq \mathfrak{m}_i$ for all $i = 1, \dots, m$.
Then $(aD:M) \nsubseteq \mathfrak{m}'$.
Hence there exists $x \in (aD:M)$ such that $x \notin \mathfrak{m}'$; that is,
for all $m \in M$, $mx \in aD$, but $x \notin \mathfrak{m}'$.
Hence $\frac{m}{1} = \frac{mx}{x} \in (aD)_{\mathfrak{m}'}$.
Therefore $(aD)_{\mathfrak{m}'} = M_{\mathfrak{m}'}$; that is,
$F_{\mathfrak{m}'} = M_{\mathfrak{m}'}$.
This fact implies that $A_{\mathfrak{m}}s \subseteq F_{\mathfrak{m}}$
for each $\mathfrak{m} \in w$-Max$(D)$,
so we have
\begin{eqnarray*}
A_ws &=& \Big(\bigcap_{\mathfrak{m} \in w\mhyphen{\rm Max}(D)} A_{\mathfrak{m}}\Big)s\\
&\subseteq & \bigcap_{\mathfrak{m} \in w\mhyphen{\rm Max}(D)} A_{\mathfrak{m}}s\\
&\subseteq & \bigcap_{\mathfrak{m} \in w\mhyphen{\rm Max}(D)} F_{\mathfrak{m}}\\
&=& F_w,
\end{eqnarray*}
where the equalities follow from \cite[Theorem 7.3.6]{wang}.
Since $F$ is finitely generated and $F \subseteq A$,
$A$ is $S$-$w$-finite type.
Thus $M$ is an $S$-SM-module.
\end{proof}

Let $D$ be an integral domain and let $M$ be a $w$-module as a $D$-module.
We say that $M$ is a {\it DW-module}
if every nonzero $D$-submodule of $M$ is a $w$-module.
Let $N_v = \{ f \in D[X] \,|\, c(f)_v = D \}$.
Then $N_v$ is a (saturated) multiplicative subset of $D[X]$ \cite[Proposition 2.1]{kang};
and the quotient module $M[X]_{N_v}$ of $M[X]$ by $N_v$ is called the {\it $t$-Nagata module} of $M$.

\begin{lemma}\label{DW-module}
Let $D$ be an integral domain
and let $M$ be a nonzero $D$-module.
Then $M[X]_{N_v}$ is a DW-module.
\end{lemma}

\begin{proof}
Suppose that $A$ is a $D[X]_{N_v}$-submodule of $M[X]_{N_v}$.
Let $f \in A_w$.
Then $fJ \in A$ for some $J \in {\rm GV}(D[X]_{N_v})$.
Note that ${\rm GV}(D[X]_{N_v}) = \{D[X]_{N_v}\}$ \cite[Theorems 6.3.12 and 6.6.18]{wang},
so $J = D[X]_{N_v}$.
Hence $f \in A$, which indicates that $A_w = A$.
Thus $M[X]_{N_v}$ is a DW-module.
\end{proof}

\begin{lemma}{\rm (cf. \cite[Proposition 6.6.13]{wang})} \label{3.7}
Let $D$ be an integral domain and
let $M$ be a torsion-free $D$-module.
Denote that $M[X]_W$ is the $w$-envelop of a $D[X]$-module $M[X]$.
Then the following assertions hold.
\begin{enumerate}
\item[(1)]
$M_w[X] = (M[X])_W$.
\item[(2)]
If $M$ is a $w$-module, then $M[X]$ is a $w$-$D[X]$-module.
\end{enumerate}
\end{lemma}

\begin{proof}
(1) Let $f := a_0 + a_1 X + \cdots + a_n X^n \in M_w[X]$.
Then $a_i \in M_w$ for all $0 \leq i \leq n$,
so for each $0 \leq i \leq n$,
there exists an element $J_i \in {\rm GV}(D)$ such that
$a_iJ_i \subseteq M$.
Let $J = J_0 \cdots J_n$.
Then $a_i J \subseteq M$ for all $0 \leq i \leq n$.
Hence $(a_0D[X] + \cdots +a_nD[X])JD[X] \subseteq M[X]$.
Since $JD[X] \in {\rm GV}(D[X])$,
$a_0D[X] + \cdots + a_nD[X] \subseteq (M[X])_W$.
It follows that $f \in (M[X])_W$.
For the reverse containment,
let $f \in (M[X])_W$.
Then there exists $J := (f_1,\dots,f_n) \in {\rm GV}(D)$ such that
$fJ \subseteq M[X]$.
Note that for each $1 \leq i \leq n$,
there exists a positive integer $m$ such that
$c(f)c(f_i)^{m+1} = c(ff_i)c(f_i)^m$ \cite[Theorem 1.7.16]{wang}.
Hence there exists a positive integer $k$ such that
$c(f)c(f_i)^{k+1} = c(ff_i)c(f_i)^k$ for all $1 \leq i \leq n$.
Therefore
\begin{center}
$c(f)(c(f_1)^{k+1} + \cdots + c(f_n)^{k+1}) =
c(ff_1)c(f_1)^k + \cdots + c(ff_n)c(f_n)^k \subseteq M$.
\end{center}
Since $c(f_1)^{k+1} + \cdots + c(f_n)^{k+1} \in {\rm GV}(D)$,
$c(f) \subseteq M_w$.
Thus $f \in M_w[X]$.
Consequently, $M_w[X] = (M[X])_W$.

(2) This is an immediate consequence of the previous result.
\end{proof}

Now, we are ready to prove the Hilbert basis theorem and
the $t$-Nagata module extension for $S$-SM-modules.

\begin{theorem}\label{main2}
Let $D$ be an integral domain,
$S$ an anti-Archimedean subset of $D$,
$N_v=\{f \in D[X] \,|\, c(f)_v=D\}$ and
$M$ a torsion-free $w$-module as a $D$-module.
Then the following statements are equivalent.
\begin{enumerate}
\item[(1)]
$M$ is an $S$-SM-module.
\item[(2)]
$M[X]$ is an $S$-SM-module.
\item[(3)]
$M[X]_{N_v}$ is an $S$-SM-module.
\item[(4)]
$M[X]_{N_v}$ is an $S$-Noetherian module.
\end{enumerate}
\end{theorem}

\begin{proof}
(1) $\Rightarrow$ (2)
First, note that $M[X]$ is a $W$-module by Lemma \ref{3.7}(2).
Let $A$ be a $w$-submodule of $M[X]$ and
let $B$ be the set consisting of zero and
the leading coefficients of the polynomials in $A$.
Then $B$ is a $D$-submodule of $M$.
Since $M$ is an $S$-SM-module,
$B$ is $S$-$w$-finite, so
there exist $s \in S$ and $b_1, \dots, b_m \in B$ such that
$Bs \subseteq (b_1D + \cdots + b_mD)_w$.
For each $i \in \{1,\dots,m\}$,
write $f_i = b_iX^{n_i} +$ (lower terms) $\in A$.
Let $n = {\rm max}\{n_1, \dots, n_m\}$ and
let $C = f_1D[X] + \cdots + f_mD[X]$.
Let $f = aX^k +$ (lower terms) $\in A$.
Then $a \in B$,
so $as \in (b_1D + \cdots + b_mD)_w$.
Therefore there exists an element $J \in {\rm GV}(D)$ such that
$asJ \subseteq b_1D + \cdots + b_mD$.
Let $J = (d_1, \dots, d_t)$.
Then for each $j \in \{1,\dots,t\}$,
$asd_j = \sum_{i=1}^{m}b_ir_{ji}$ for some $r_{j1}, \dots, r_{jm} \in D$.
If $k \geq n$, then for each $j \in \{1,\dots,t\}$,
let $g_j = fsd_j - \sum_{i=1}^{m}  f_ir_{ji}X^{k-n_i}$.
Then for all $j \in \{1,\dots,t\}$, $g_j \in A$ with $\deg(g_j) <k$.
If we still have some $j \in \{1,\dots,t\}$ such that
$\deg(g_j)\geq n$, then we repeat the same process.
Let $b$ be the leading coefficient of $g_j$.
Then $b \in B$,
so $bsJ_1 \subseteq b_1D + \cdots + b_m D$ for some
$J_1:=(d_1',\dots, d_{t'}') \in {\rm GV}(D)$.
Hence for each $\ell \in \{ 1,\dots, t'\}$,
$bsd_{\ell}' = \sum_{i=1}^{m}b_ir'_{\ell i}$.
Let $g_j' = g_jsd_{\ell}' - \sum_{i=1}^{m}f_ir_{\ell i}'X^{\deg(g_j)-n_i}$.
Then $g'_j \in A$, $\deg(g_j') < \deg(g_j)$ and
$g_j' = (fsd_j - \sum_{i=1}^{m}  f_i r_{ji}X^{k-n_i})sd_{\ell}' -
\sum_{i=1}^{m}f_ir_{\ell i}'X^{\deg(g_j)-n_i}$.
After finitely many steps,
we get $J' \in {\rm GV}(D)$ and an integer $q \geq 1$ such that
$fs^qJ' \subseteq (A \cap L) +C$, where $L = M \oplus MX \oplus \cdots \oplus MX^{n-1}$.
Since $L$ is an $S$-SM-module \cite[Lemma 2.7(2)]{kim},
$(A\cap L)_{w}$ is $S$-$w$-finite, so
there exist $t \in S$ and $h_1, \dots , h_s \in A \cap L$ such that
$(A \cap L)_{w}t \subseteq (h_1D + \cdots + h_sD)_{w} \subseteq
(h_1D[X] + \cdots + h_sD[X])_{W}$.
Let $u \in \bigcap_{i\geq 1}s^iD \cap S$.
Then we have
\begin{eqnarray*}
futJ'D[X] &\subseteq & ((A \cap L) +C)tD[X]\\
&=& (A \cap L)tD[X] + CtD[X]\\
&\subseteq & (h_1D[X] + \cdots + h_sD[X])_{W} + C.
\end{eqnarray*}
Since $J'D[X] \in {\rm GV}(D[X])$, $fut \in ((h_1D[X] + \cdots + h_sD[X])_{W} + C)_{W}$.
Hence we have
\begin{eqnarray*}
Aut &\subseteq & ((h_1D[X] + \cdots + h_sD[X])_{W} + C)_{W}\\
&=& (h_1D[X] + \cdots + h_sD[X] + C)_{W}.
\end{eqnarray*}
Thus $A$ is $S$-$w$-finite.
Consequently, $M[X]$ is an $S$-SM-module.

(2) $\Rightarrow$ (4)
Let $A$ be a nonzero $D[X]_{N_v}$-submodule of $M[X]_{N_v}$.
Then by Lemma \ref{quotient 1}(1),
$A = A'_{N_v}$ for some nonzero $D[X]$-submodule $A'$ of $M[X]$.
Since $A'_w$ is $S$-$w$-finite,
there exist $s \in S$ and $f_1, \dots, f_n \in A'$ such that
$A'_ws \subseteq (f_1D[X] + \cdots + f_nD[X])_w$.
Let $f \in A$.
Then $fg \in A'$ for some $g \in N_v$,
so we have
\begin{center}
$fgsJ \subseteq f_1D[X] + \cdots + f_nD[X]$
\end{center}
for some $J \in {\rm GV}(D[X])$.
Write $J = (h_1, \dots, h_m)$ for some $h_1, \dots, h_m \in D[X]$
and let $h = h_1 + h_2X^{\deg(h_1) +1} + \cdots + h_mX^{\sum_{j=1}^{m}\deg(h_j) + m-1}$.
Then $c(h)_v = D$ and $fsgh \in f_1 D[X] + \cdots + f_nD[X]$.
Since $gh \in N_v$, we obtain
\begin{center}
$fs  \in (f_1D[X] + \cdots + f_nD[X])_{N_v}$.
\end{center}
Hence $As \subseteq (f_1 D[X] + \cdots + f_n D[X])_{N_v}$.
Thus $M[X]_{N_v}$ is an $S$-Noetherian $D[X]_{N_v}$-module.

(4) $\Rightarrow$ (1)
Let $A$ be a $w$-submodule of $M$.
Then $A[X]_{N_v}$ is a $D[X]_{N_v}$-submodule of $M[X]_{N_v}$.
Since $M[X]_{N_v}$ is an $S$-Noetherian module,
there exist $s \in S$ and $f_1, \dots, f_n \in A[X]$ such that
$A[X]_{N_v}s \subseteq (f_1D[X] + \cdots + f_nD[X])_{N_v}$, so we obtain
\begin{center}
$A[X]_{N_v}s \subseteq (c(f_1) + \cdots + c(f_n))[X]_{N_v}$.
\end{center}
Let $a \in A$.
Then $asg \in (c(f_1) + \cdots + c(f_n))[X]$ for some $g \in N_v$, so
$asc(g) \subseteq c(f_1) + \cdots + c(f_n)$.
Since $c(g) \in {\rm GV}(D)$,
$as \in (c(f_1) + \cdots + c(f_n))_w$ \cite[Proposition 3.5]{ywzc} (or \cite[Lemma 2.4]{hl}).
Hence $As \subseteq (c(f_1) + \cdots + c(f_n))_w$.
Thus $A$ is $S$-$w$-finite.
Consequently, $M$ is an $S$-SM-module.

(3) $\Leftrightarrow$ (4)
This equivalence comes directly from Lemma \ref{DW-module}.
\end{proof}

The following result has already been proved in \cite{kim},
but we can prove it in a different way from the proof in \cite{kim} using Theorem \ref{main2}.

\begin{corollary}{\rm (\cite[Theorem 2.11(2)]{kim})}
Let $D$ be an integral domain and
let $S$ be a multiplicative subset of $D$.
Then $D$ is an $S$-SM-domain if and only if
every $S$-$w$-finite torsion-free $w$-module
is an $S$-SM-module.
\end{corollary}

\begin{proof}
Suppose that $D$ is an $S$-SM-domain and
let $M$ be an $S$-$w$-finite torsion-free $w$-module as a $D$-module.
Then there exist $s \in S$ and
a finitely generated $D$-submodule $L$ of $M$
such that $Ms \subseteq L_w$,
so $M[X]_{N_v}s \subseteq L_w[X]_{N_v} = L[X]_{N_v}$,
where the equality comes from \cite[Lemma 2.4(3)]{chang1}.
Hence $M[X]_{N_v}$ is $S$-finite.
Since $D[X]_{N_v}$ is an $S$-Noetherian domain \cite[Theorem 2.8]{kim},
$M[X]_{N_v}$ is an $S$-Noetherian $D[X]_{N_v}$-module \cite[Proposition 2.1]{hamed1}.
Thus by Theorem \ref{main2}, $M$ is an $S$-SM-module.
The converse is obvious.
\end{proof}

The next result recovers the fact that
every surjective endomorphism of an SM-module
is an isomorphism \cite[Theorem 2.10]{chang1}.

\begin{proposition}
Let $D$ be an integral domain,
$S$ a multiplicative subset of $D$ and
$M$ a torsion-free $w$-module as a $D$-module.
If $M$ is an $S$-SM-module and
$\varphi : M \to M$ is a $D$-module epimorphism,
then $\varphi$ is an isomorphism.
\end{proposition}

\begin{proof}
For each $n \geq 2$, let $\varphi^n = \varphi^{n-1} \circ \varphi$.
Then $\varphi^n$ is a $D$-module homomorphism, so
${\rm Ker}(\varphi^n)$ is a $w$-submodule of $M$
for all $n \geq 2$ \cite[Lemma 2.9]{chang1}.
Hence we obtain an ascending chain
${\rm Ker}(\varphi) \subseteq {\rm Ker}(\varphi^2) \subseteq \cdots$
of $w$-submodules of $M$.
Since $M$ is an $S$-SM-module,
there exist $s \in S$ and $k \in \mathbb{N}$ such that
${\rm Ker}(\varphi^k)s \subseteq {\rm Ker}(\varphi^n)$ for all $k \geq n$ \cite[Theorem 1]{hamed}.
Let $x \in {\rm Ker}(\varphi)$.
Since $\varphi^n$ is surjective,
there exists an element $m \in M$ such that $\varphi^n(m) = x$, so
$\varphi^{n+1}(m) = \varphi(x) = 0$.
Therefore $m \in {\rm Ker}(\varphi^{n+1})$,
which implies that $ms \in {\rm Ker}(\varphi^n)$.
Hence $\varphi^n(m)s = \varphi^n(ms) = 0$.
Since $M$ is torsion-free, $\varphi^n(m) = 0$.
Thus $\varphi$ is an isomorphism.
\end{proof}

\end{document}